\documentclass[10pt,a4paper]{article}

\usepackage{latexsym}
\usepackage{verbatim}
\usepackage{amsthm}
\usepackage{amsfonts}
\usepackage{amsmath}
\usepackage{amssymb}
\usepackage{amscd}

\numberwithin{equation}{section}

\newtheorem{prop}{Proposition}[section]
\newtheorem{theorem}[prop]{Theorem}
\newtheorem{lemma}[prop]{Lemma}

\newtheorem{remark}[prop]{Remark}
\newtheorem{example}[prop]{Example}
\newtheorem{definition}[prop]{Definition}

\def\begeq{\begin{equation}}
\def\endeq{\end{equation}}

\def\<{\langle}
\def\>{\rangle}
\def\({\left(}
\def\){\right)}

\def\pa{\partial}
\def\ep{\epsilon}

\def\ep{\varepsilon}

\def\lbr{\lbrace}
\def\rbr{\rbrace}
\def\rn{\mathbb R^n}

\begin{document}

\title{Three Bernstein type theorems for hypersurfaces with zero Gaussian curvature}
\author{S\l awomir Dinew, Mengru Guo, Heming Jiao }

\date{}
\maketitle

\begin{abstract}
In this paper, we prove Bernstein type theorems for entire convex graphical hypersurfaces with zero Gaussian curvature in both
Euclidean and Minkowski context.
A supplementary example illustrates that zero Gaussian convex spacelike
hypersurfaces are not necessary hyperplanes without  additional conditions.
We show that a zero Gaussian curvature convex hypersurface must be a hyperplane if the mean curvature goes to zero at infinity.
In the Minkowski context, we prove similar results for hypersurface without timelike points.

\end{abstract}
\section{Introduction}

It is well known that an entire graphic minimal hypersurface in $\mathbb{R}^{n+1}$ for $n \leq 7$ is a hyperplane. This result is called Bernstein theorem
since Bernstein proved it for $n=2$. (See \cite{A,B,De,Fl,S}.) Examples of Bombieri-De Giorgi-Giusti \cite{BDG} show that that the Bernstein theorem is no longer valid
for $n \geq 8$ without any additional conditions. Ecker-Huisken \cite{EH} extended Bernstein theorem to all dimensions under an additional sub-linear gradient growth condition.

Unlike the Euclidean case, in the Minkowski context, an entire spacelike graph with zero mean curvature is always a hyperplane for all dimensions.
(See the proof of Calabi \cite{Calabi} for $n \leq 4$ and Cheng-Yau \cite{CY} for general dimensions.)
Recently, Akamine-Honda-Umehara-Yamada \cite{AHUY} showed that an entire zero mean curvature graph which admits no timelike  points is a hyperplane.

It is natural to ask if an entire (convex) graphic hypersurface with zero Gaussian curvature is a hyperplane (both in Euclidean and Minkowski context).
To answer this question, we need to study the rigidity of the degenerate Monge-Amp\`{e}re equation
\begin{equation}
\label{MA}
\det D^2 u = 0 \ \ \mbox{ in } \mathbb{R}^n.
\end{equation}
In particular, we want to know whether entire convex solutions to \eqref{MA} are affine. The following example shows that the answer is negative without additional conditions.
Furthermore, the example shows that uniform gradient bound  is not enough for the affinity of the solution.
\begin{example}
\label{example}
Let
\[
u (x) = a \sqrt{x_1^2 + c}\  \ \ (c > 0), \mbox{ for } x = (x_1, x_2) \in \mathbb{R}^2,
\]
where $0 < a < 1$ is a constant. It is easy to see that $\|Du\| < 1$ (which means that the graph of $u$ is spacelike everywhere in the Minkowski context
by Definition \ref{space_light_time_like} below) and that $u$ is convex in $\mathbb{R}^2$.
Obviously, $u$ is a solution to \eqref{MA} which is not affine.
\end{example}
The purpose of this paper is to explore additional conditions under which entire convex hypersurfaces with zero Gaussian curvature are hyperplanes,
or, equivalently, convex solutions to \eqref{MA} are affine.

Traditionally, growth conditions on the solution are imposed (see, for example, \cite{BCGJ}). Note however, that Example \ref{example} has exactly the expected linear growth. It even satisfies the stronger pointwise lower bound
$$u(x)\geq l(x)$$
for a suitable linear function. Assuming, in turn, pointwise upper bounds $u(x)\leq l(x)$ trivializes matters in our setting. Thus, growth at infinity does not seem to be an appropriate condition. Moving on to gradient (or first order) bounds we note that Example \ref{example} has uniformly bounded gradient. Hence it seems natural to ask for a {\it second order condition} i.e., an assumption on the asymptotic of the Hessian of the solution. We prove such rigidity results assuming bounds on the mean curvature of the hypersurface or the Laplacian of $u$. The first, more analytic statement, reads as follows:
\begin{theorem}
\label{thm_analysis}
Let $u \in C^2 (\mathbb{R}^n)$ be a convex solution to \eqref{MA} satisfying
\begin{equation}
\label{condition_analysis}
\lim_{x \rightarrow \infty} \Delta u (x) = 0.
\end{equation}
Then $u$ is an affine function.
\end{theorem}
Roughly speaking, a decay at infinity of a certain natural combination of second order derivatives forces the affinity.

The same idea will be used in our other results which have more geometric flavor. To this end we shall assume a decay of the mean curvature of the hypersurface. Before we state the precise result, we will introduce the notations and definitions
first.

\subsection{Hypersurfaces in the Euclidean space $\mathbb{R}^{n+1}$}

Given a function $u: \mathbb{R}^n \rightarrow \mathbb{R}$ let $M_u=\{(x,u(x)): x\in\mathbb{R}^n\}$ be the graphic hypersurface
defined by $u$. We consider $M_u$ as a submanifold in the Euclidean space $\mathbb R^{n+1}$. By direct calculations, the induced metric and second fundamental form of $M_u$ are given
by
$$g^E_{ij}=\delta_{ij}+ u_i u_j, \ \  1\leq i,j\leq n,$$
(here the symbol $E$ stands for the Euclidean setting) and
\[h^E_{ij}=\frac{u_{ij}}{\sqrt{1+|Du|^2}}\]
respectively. The principal curvatures of $M_u$ are the eigenvalues of the matrix
\[
\frac{1}{\sqrt{1+\|Du\|^2}} \left(I - \frac{Du \otimes Du}{1+\|Du\|^2}\right) D^2 u,
\]
where $I$ is the $n\times n$ unit matrix. Therefore, the mean and Gaussian curvatures of $M_u$ are
\[
H^E_u = (w^E)^{-3/2} (w^E \Delta u - u_i u_j u_{ij})
\]
and
\[
K^E_u = (w^E)^{-(n+2)/2} \det D^2 u
\]
respectively, where
\[
w^E = 1 + \|Du\|^2.
\]
Our next result can be stated as follows:
\begin{theorem}
\label{rigidity_Euclidean1}
Let $u \in C^4 (\mathbb{R}^n)$ be convex and the Gaussian curvature of $M_u$, $K^E_u$ vanishes identically. Assume furthermore that
\begin{equation}
\label{condition_E}
\lim_{x \rightarrow \infty} H^E_u (x) = 0.
\end{equation}
Then $M_u$ is a hyperplane.
\end{theorem}

\begin{remark}
 As $M_u$ is a graph of a convex function, we always have $H^E_u(x)\geq 0$. Thus our additional condition is that the graph of $u$ flattens at infinity in the sense of decaying mean curvature.
\end{remark}

\subsection{Hypersurfaces in the Minkowski space $\mathbb{R}^{n,1}$}

Let $\mathbb{R}^{n,1}$ be the Minkowski space, i.e., the space $\mathbb{R}^n \times \mathbb{R}$ equipped with the metric
\[
ds^2=dx^2_1+\cdots+dx^2_n-dx^2_{n+1}.
\]
\begin{definition}
\label{space_light_time_like}
Let $u: \mathbb{R}^n \rightarrow \mathbb{R}$ be a function. $x \in \mathbb{R}^n$ is called a spacelike (resp. lightlike, timelike) point of
$M_u$ if $\|Du (x)\| < 1$ (resp. $\|Du (x)\| = 1$, $\|Du (x)\| > 1$). The graph of $u$ is called spacelike (resp. lightlike, timelike) if it consists entirely of spacelike (resp. lightlike, timelike points).
\end{definition}
If the graph $M_u$ is spacelike, the induced metric and second fundamental form of $M$ are given
by
$$g^M_{ij}=\delta_{ij} - u_i u_j, \ \  1\leq i,j\leq n,$$
(with the letter $M$ indicating the Minkowski setting) and
\[h^M_{ij}=\frac{u_{ij}}{\sqrt{1-|Du|^2}}\]
respectively. The principal curvatures of $M_u$ are the eigenvalues of the matrix
\[
\frac{1}{\sqrt{1-\|Du\|^2}} \left(I + \frac{Du \otimes Du}{1-\|Du\|^2}\right) D^2 u,
\]
and the mean and Gaussian curvatures of $M_u$ are
\[
H^M_u = (w^M)^{-3/2} (w^M \Delta u + u_i u_j u_{ij})
\]
and
\[
K^M_u = (w^M)^{-(n+2)/2} \det D^2 u
\]
respectively, where
\[
w^M = 1 - \|Du\|^2.
\]

In order to make the definitions be compatible with the case that $M_u$ may admit lightlike points, we define
\begin{equation}\label{mean_gauss}
\tilde{H}^M_u = w^M \Delta u + u_i u_j u_{ij}
\end{equation}

and
\[
\tilde{K}^M_u = (w^M)^{(n+2)/2} K^M_u = \det D^2 u.
\]
\begin{theorem}
\label{rigidity_Minkowski}
Let $u \in C^4 (\mathbb{R}^n)$ be convex and $\tilde{K}^M_u \equiv 0$. Assume that $M_u$ admits no timelike points and that
\begin{equation}
\label{condition_M}
\lim_{x \rightarrow \infty} \tilde{H}^M_u (x) = 0.
\end{equation}
Then $M_u$ is a hyperplane.
\end{theorem}
\begin{remark}
For entirely spacelike Gaussian-flat, convex graphical hypersurfaces the decay of the mean curvature $H^M_u$ at infinity also leads to affinity - the proof of this fact is completely analogous and is hence omitted.
\end{remark}

\subsection{Historical remarks}

Below we recall some facts regarding the rigidity of entire solutions to the non-degenerate Monge-Amp\`{e}re equation
\begin{equation}
\label{nondegenerateMA}
\det D^2 u = 1 \mbox{  in } \mathbb{R}^n
\end{equation}
and other fully nonlinear equations. Unlike the degenerate case \eqref{MA}, the solutions to \eqref{nondegenerateMA} are
always quadratic polynomials. Indeed, each convex entire solution to \eqref{nondegenerateMA} is
a quadratic polynomial. It was proved by J\"{o}rgens \cite{J} for $n=2$, Calabi \cite{Calabi1} for $n=3,4,5$
and by Pogorelov \cite{P} for general $n$. Then, Cheng-Yau \cite{CY2} provided anther more geometrical proof. Remarkably, this result, known as Calabi-J\"orgens-Pogorelov theorem requires no additional assumptions on the solution $u$.

It is worth emphasizing that for other nonlinear operators the corresponding rigidity statement becomes highly nontrivial. In particular, Chang-Yuan \cite{CYu} proposed that entire solutions to the $k$-Hessian equations
\begin{equation}
\label{sigma_k}
\sigma_k (\lambda (D^2 u)) = 1
\end{equation}
where $\sigma_k(\lambda)$ are the elementary symmetric functions
\[
\sigma_k (\lambda) = \sum_{1\leq i_1 < \cdots < i_k \leq n} \lambda_{i_1} \cdots \lambda_{i_k}, \ k= 1, \ldots, n.
\]
 must be quadratic polynomials if they satisfy  quadratic lower bounds. They proved the case $k=2$ with an additional convexity condition
\[
D^2 u \geq \delta - \sqrt{\frac{2n}{n-1}}.
\]

Bao-Chen-Guan-Ji \cite{BCGJ} proved the rigidity for strictly convex entire solutions to \eqref{sigma_k}. Their result states that solutions $u$ to (\ref{sigma_k}) satisfying a quadric growth are
quadratic polynomials. Then Li-Ren-Wang \cite{LRW} generalized their results to $(k+1)$-convex solutions. Recently, Chu and the first author \cite{CD}
generalized the above results to more general Hessian operators. Another class of equations named the special Lagrangian equations were studied by Yuan - \cite{Y}. In \cite{Y} a Bernstein theorem is proved for minimal surface of the form $(x, \nabla u) \subset \mathbb{R}^n \times \mathbb{R}^n$ with $u$ being a smooth convex function
in $\mathbb{R}^n$. It is of interest to consider the rigidity of the graphic hypersurface $M_u$ satisfying
\[
\sigma_k (\kappa (M_u)) = 0 \ \mbox{ in } \mathbb{R}^n,
\]
or more general curvature conditions under the restriction \eqref{condition_E} (or \eqref{condition_M}), where $\kappa (M_u)$ denote
the principal curvatures of $M_u$.

{\bf Acknowledgements}. The first named author is supported by by Sheng
grant no. 2023/48/Q/ST1/00048 of the National Science Center, Poland. The work was carried out while the second named author was
visiting the Faculty of Mathematics and Computer Science at Jagiellonian University. She wishes to thank the Faculty and the
University for their hospitality. She also would like to thank China Scholarship Council for their support.
The third named author is supported by the National Natural Science Foundation of China (Grant No. 12271126).

\section{Proofs}
In this section we provide the proofs of the statements in the introduction.

First, we recall what is meant by {\it affinity} of a (convex) function along a subset:
\begin{definition}\label{def:affinity_set}
Let $U$ be a convex subset of $\rn$ and $u: U\longmapsto \mathbb R$ be a convex function, which is continuous up to $\overline{U}$. Then $f$ is said to be affine along $L\subset{\overline{U}}$ if for any $x,y\in L$ so that the segment $[x,y]$ is contained in $L$ we have for any $t\in(0,1)$
$$u(tx+(1-t)y)\leq tu(x)+(1-t)u(y).$$
\end{definition}

The next result is Lemma 2 from \cite{CNS}:

\begin{lemma}\label{lem:Caffarelli_Nirenberg_Spruck}
 Let $U$ be a convex subset of $\rn$ and $u: U\longmapsto \mathbb R$ be a convex function, which is continuous up to $\overline{U}$. Suppose $u$ additionally satisfies
 $$\det(D^2u)=0.$$
 Then, for any $x_0\in U$ there are (at most) $n+1$ points $x_1,\cdots,x_m\in\partial U$ such that $x_0$ belongs to the closed simplex $L$ with vertices $x_1,\cdots x_k$ and $u$ is affine on $L$.
\end{lemma}

We remark that the result also holds for non-smooth convex $u$ - then the equality $\det(D^2u)=0$ has to be understood in the weak sense.

The next two lemmas are elementary but crucial for the argument. The first one is inspired by the computations in \cite{BF} and \cite{BB}, where solutions of the {\it complex} Monge-Amp\`ere equations were analyzed. The real counterpart can be found in \cite{F}.

\begin{lemma}\label{lem:convexity_in_affinity_directions}
 Let $u: U\longmapsto \mathbb R$ be a convex function of class $C^4$, which is affine an a set $L\subset U$. Suppose that for some $x_0\in U$, $\ep>0$ and some unit vector $\gamma\in\rn$ the segment $[x_0-\ep\gamma,x_0+\ep\gamma]$ is contained in $L$.
 Then for any vector $\eta\in\rn$ we have
 $$u_{\eta\gamma\gamma}(x_0)=0,\ \ u_{\eta\eta\gamma\gamma}(x_0)\geq 0.$$
 More generally, for $u \in C^2$ (and still convex) the function
 $$(-\ep,\ep)\ni t\longmapsto u_{\eta\eta}(x_0+t\gamma)$$
is convex.
\end{lemma}
\begin{proof}
 As $[x_0-\ep\gamma,x_0+\ep\gamma]$ is compact, it admits a neighborhood compactly supported in $U$. thus, for sufficiently small $\delta>0$ the function
 $f:(-\delta,\delta)\ni t\longmapsto u_{\gamma\gamma}(x_0+t\eta)$ is well defined. Note that, by convexity, $f\geq 0$ and $f(0)=0$ as $u$ is affine along the segment. So $f$ has a local minimum at $0$. Hence $$u_{\eta\gamma\gamma}(x_0)=f'(0)=0,\ \  u_{\eta\eta\gamma\gamma}(x_0)=f''(0)\geq 0,$$
 as claimed.

 To get the claim for $u\in C^2$ note that
 $$u_{\eta\eta}(x_0+t\gamma)=lim_{s\rightarrow 0^+}\frac1{s^2}[u(x_0+t\gamma+s\eta)+u(x_0+t\gamma-s\eta)-2u(x_0+t\gamma)],$$
 (the right hand side function is well-defined for sufficiently small $s>0$). But for each fixed $s>0$ the function
 $$t\longmapsto [u(x_0+t\gamma+s\eta)+u(x_0+t\gamma-s\eta)-2u(x_0+t\gamma)]$$
 is convex as the the sum of two convex functions minus an affine one. Taking the limit concludes the proof.
\end{proof}
The second lemma will be very helpful in the computations later on:
\begin{lemma}\label{lem:gradient_derivative_vanishing}
 Let $u: U\longmapsto \mathbb R$ be a convex function of class $C^3$, which is affine an a set $L\subset U$. Suppose that for some $x_0\in U$, $\ep>0$ and some unit vector $\gamma\in\rn$ the segment $[x_0-\ep\gamma,x_0+\ep\gamma]$ is contained in $L$. Then
 $$[\|Du(x_0)\|^2]_{\gamma}=[\|Du(x_0)\|^2]_{\gamma\gamma}=0.$$
\end{lemma}
\begin{proof}
 Complete $\gamma=:\gamma_1$ to an orthonormal basis $(\gamma_1,\cdots,\gamma_n)$ of $\rn$. Then $||Du(x)||^2=\sum_{k=1}^n(u_{\gamma_k}(x))^2$. The matrix $(u_{\gamma_j\gamma_j}(x_0))$ is positive semi-definite as $u$ is convex. The vanishing of $u_{\gamma_1\gamma_1}(x_0)$ (which follows from the affinity of $u$ along the segment $[x_0-\ep\gamma,x_0+\ep\gamma]$) forces then the vanishing of each $u_{\gamma_1\gamma_j}(x_0)$.

 But then
 \begin{equation*}
  \|Du(x_0)\|^2_\gamma=2\sum_{k=1}^nu_{\gamma_k}(x_0)u_{\gamma_k\gamma}(x_0)=0.
 \end{equation*}
Analogously
\begin{equation*}
 \|Du(x_0)\|^2_{\gamma\gamma}=2\sum_{k=1}^nu_{\gamma_k\gamma}(x_0)u_{\gamma_k\gamma}(x_0)
 +2\sum_{k=1}^n2u_{\gamma_k}(x_0)u_{\gamma_k\gamma\gamma}(x_0)=0,
\end{equation*}
where we used Lemma \ref{lem:convexity_in_affinity_directions} to justify the vanishing of the second term.
\end{proof}
The next proposition is the technical heart of the argument:
\begin{prop}\label{prop:convexity_of_H_along_affine_sets_of_u}
Let $u: U\longmapsto \mathbb R$ be a convex function of class $C^4$, which is affine an a set $L\subset U$. Then the function $\tilde{H}^M_u (x)$
(defined in \eqref{mean_gauss}) is convex along $L$.
\end{prop}
\begin{proof}
 Suppose that for some $x_0\in U$, $\ep>0$ and some unit vector $\gamma\in\rn$ the segment $[x_0-\ep\gamma,x_0+\ep\gamma]$ is contained in $L$. It suffices to show that $(\tilde{H}^M_u)_{\gamma\gamma}(x_0)\geq 0$.
By calculations, we have
\[
\begin{aligned}
 & (\tilde{H}^M_u)_{\gamma\gamma}(x_0)=w^M[\Delta u(x_0)]_{\gamma\gamma} + 2 (1-\|Du(x_0)\|^2)_\gamma [\Delta u(x_0)]_{\gamma}\\
 &+ \Delta u(x_0)(1-\|Du(x_0)\|^2)_{\gamma\gamma}
 + [u_k(x_0)u_l(x_0)u_{kl}(x_0)]_{\gamma\gamma}\\
 & = I+II+III+IV.
\end{aligned}
\]
A direct application of Lemma \ref{lem:gradient_derivative_vanishing} implies that all terms, except for $I$ and $IV$, vanish.
Yet another application of the same lemma results in
$$IV=u_k(x_0)u_l(x_0)u_{kl\gamma\gamma}(x_0).$$
Application of Lemma \ref{lem:convexity_in_affinity_directions} for $\eta=e_k,\ k=1,\cdots,n$ (each of the coordinate directions) results in the bound $I\geq 0$. The same lemma applied to $\eta=(u_1(x_0),\cdots,u_n(x_0))$ implies $IV\geq 0$.
Hence $(\tilde{H}^M_u)_{\gamma\gamma}(x_0)=I+IV\geq 0$, as claimed.
\end{proof}

\begin{proof}[Proof of  Theorem \ref{rigidity_Minkowski}]
Fix $\ep>0$ and a point $x_0\in\rn$. The assumptions on the decay of $\tilde{H}^M_u$ imply the existence of a large radius $R>1$ so that
$$\forall x\in \pa B_R(x_0)\ \ 0\leq \tilde{H}^M_u (x)\leq \ep.$$
From Lemma \ref{lem:Caffarelli_Nirenberg_Spruck} there are at most $n+1$ points $x_1,\cdots, x_m\in\pa B_R(x_0)$ so that $x_0$ belongs to the closed simplex $L$ spanned by $x_i$ and $u$ is affine along $L$. Then Proposition \ref{prop:convexity_of_H_along_affine_sets_of_u} forces the convexity of $\tilde{H}^M_u$ along $L$. In particular,
$$\tilde{H}^M_u(x_0)\leq max\lbr \tilde{H}^M_u(x_1),\cdots,\tilde{H}^M_u(x_m)\rbr\leq\ep.$$
As $\ep$ is arbitrary $\tilde{H}^M_u (x_0)=0$. Convexity of $u$ forces each principal curvature at $x_0$ to vanish. Thus the graph of $u$ is affine, as claimed.
\end{proof}

\begin{proof}[Proof of Theorem \ref{thm_analysis}]
 By the last remark in Lemma \ref{lem:convexity_in_affinity_directions} $\Delta u$ is convex along the affinity set of $u$. Hence, just as in the previous proof, the decay of $\Delta u$ at infinity leads to global vanishing. Finally, harmonic convex functions have to be affine. This finishes the proof.
\end{proof}

\begin{proof}[Proof of Theorem \ref{rigidity_Euclidean1}]
	
	Once again, the core of the proof is to show the convexity of $H^E_u$ in the affinity directions of $u$. Just as in the proof of Proposition \ref{prop:convexity_of_H_along_affine_sets_of_u} fix a point $x_0\in\mathbb R^n$ and a direction $\gamma\in\mathbb R^n$, so that $u$ is affine on $[x_0-\varepsilon\gamma, x_0-\varepsilon\gamma]$ for some $\varepsilon>0$.

Analogous computation shows that
$$(H^E_u(x_0))_{\gamma\gamma}=(w^E(x_0))^{-3/2}[w^E(x_0)\Delta u(x_0)_{\gamma\gamma}-u_i(x_0)u_j(x_0)u_{ij\gamma\gamma}(x_0)].$$

It suffices to establish the non-negativity of the term in the square brackets. To this end for each tuple $(i,j), i<j$ we define the vector $\eta_{ij}\in\mathbb R^n$ with
coefficients $u_j(x_0)$ on the $i$-th slot, $-u_i(x_0)$ on the $j$-th slot and zero otherwise. Then,
$$w^E(x_0)\Delta u(x_0)_{\gamma\gamma}-u_i(x_0)u_j(x_0)u_{ij\gamma\gamma}(x_0)=\Delta u(x_0)_{\gamma\gamma}+\sum_{i<j}u_{\eta_{ij}\eta_{ij}\gamma\gamma}(x_0),$$
which is non-negative by Lemma \ref{lem:convexity_in_affinity_directions}. Hence $(H^E_u(x_0))_{\gamma\gamma}\geq0$, as claimed.
\end{proof}

\begin{remark}
The smoothness  assumption is essential in all theorems above as the following example shows: let $v(x)=a||x||$ for any $a\in(0,1)$. Then $det(D^2v)=0$ on $\mathbb R^n\setminus\lbrace0\rbrace$, while $\Delta v(x)=a\frac{n-1}{||x||}$, $H^E_v(x)=\frac{(n-1)a}{\sqrt{(1+a^2)}||x||}$, $\tilde{H}^M_v(x)=(1-a^2)\frac{(n-1)a}{||x||}$. For any $1\leq k<n$ the function
$v_k(x):=a\sqrt{\sum_{j=1}^kx_j^2},\ a\in(0,1)$ satisfies even globally  $det(D^2v_k)=0$ in a weak sense and all $\Delta v_k, H^E_{v_k}, \tilde{H}^M_{v_k}$ decay to zero in the smooth directions.
\end{remark}

Faculty of Mathematics and Computer Science,  Jagiellonian University 30-348 Krakow, Lojasiewicza 6, Poland;\\ e-mail: {\tt slawomir.dinew@im.uj.edu.pl}

School of Mathematics, Harbin Institute of Technology, Harbin, Heilongjiang 150001, China;\\ e-mail: {\tt 22B912007@stu.hit.edu.cn}

School of Mathematics and Institute for Advanced Study in Mathematics, Harbin Institute of Technology, Harbin, Heilongjiang 150001, China;\\ e-mail: {\tt jiao@hit.edu.cn}

\end{document}